\documentclass[12pt]{article}
\usepackage[a4paper,margin=0.8in]{geometry}

\usepackage[colorlinks,citecolor=magenta,linkcolor=black]{hyperref}
\pdfpagewidth=\paperwidth \pdfpageheight=\paperheight
\usepackage{amsfonts,amssymb,amsthm,amsmath,eucal,tabu,url}
\usepackage{pgf}
 \usepackage{array}
 \usepackage{tikz-cd}
 \usepackage{pstricks}
 \usepackage{pstricks-add}
 \usepackage{pgf,tikz}
 \usetikzlibrary{automata}
 \usetikzlibrary{arrows}
 \usepackage{indentfirst}
\usepackage{tabularx} 

\theoremstyle{plain}
\newtheorem{thm}{Theorem}[section]
\newtheorem{theorem}[thm]{Theorem}
\newtheorem*{theoremA}{Main Theorem}

\newtheorem{proposition}[thm]{Proposition}
\newtheorem{corollary}[thm]{Corollary}

\theoremstyle{definition}
\newtheorem{definition}[thm]{Definition}
\newtheorem{remark}[thm]{Remark}

\newtheorem{problem}[thm]{Problem}

\newtheorem{thevarthm}[thm]{\varthmname}

\newenvironment{varthm*}[1]{\trivlist\item[]{\bf #1.}\it}{\endtrivlist}

\def\keywordname{{\bfseries Keywords}}%
\def\keywords#1{\par\addvspace\medskipamount{\rightskip=0pt plus1cm
\def\and{\ifhmode\unskip\nobreak\fi\ $\cdot$
}\noindent\keywordname\enspace\ignorespaces#1\par}}
\def\subclassname{{\bfseries Mathematics Subject Classification
(2020)}\enspace}
\def\subclass#1{\par\addvspace\medskipamount{\rightskip=0pt plus1cm
\def\and{\ifhmode\unskip\nobreak\fi\ $\cdot$
}\noindent\subclassname\ignorespaces#1\par}}

\begin{document}
\title{Minimal plus-one generated line arrangements with double and triple intersection points}
\author{Artur Bromboszcz}
\date{\today}
\maketitle
\thispagestyle{empty}
\begin{abstract}
We provide a complete classification, in the language of weak-combinatorics, of minimal plus-one generated line arrangements in the complex projective plane with double and triple intersection points.
\subclass{minimal plus-one generated curves, line arrangements, intersection points, weak-combinatorics}
\end{abstract}
\section{Introduction}
In a very recent paper \cite{A0}, Abe introduced a new class of hyperplane arrangements, namely the class of plus-one generated  arrangements. It turns our that these arrangement are strongly connected with free hyperplane arrangements and we are going to explain it shortly now. An arrangement of hyperplanes is called {\it free} if its associated module of derivations is a free module over the polynomial ring. For plus-one generated arrangements the associated module of derivations is no longer free, but it admits a simple minimal free resolution, so it is a natural step away from freeness property. Moreover, if we restrict our attention to line arrangements in the projective plane, the plus-one generated property appears in close relation to the freeness, namely if we delete a line from a free line arrangement, then the resulting arrangement is either free or plus-one generated, see also \cite{A0} for details. 

Here in the paper we focus on the classification problem of line arrangements in the complex projective plane that admit only double and triple intersection points and we provide a complete classification of plus-one generated line arrangements having only double and triple points that depends on the weak-combinatorics of such arrangements. Let $\mathcal{L} \subset \mathbb{P}^{2}_{\mathbb{C}}$ be an arrangement of $d$. Denote by $n_{i} = n_{i}(\mathcal{L})$ the number of $i$-fold intersection points among the lines in $\mathcal{L}$. Then the weak-combinatorics of a given line arrangement $\mathcal{L}$ is defined as
$$W(\mathcal{L}) = (d; n_{2}, \dots ,n_{d}),$$
and we will use a convention that we truncate the vector $W(\mathcal{L})$ by removing data $n_{i} = 0$ for $i > m(\mathcal{L})$, where $m(\mathcal{L})$ denotes the maximal multiplicity of intersection points in $\mathcal{L}$. 

Also very recently, Dimca and Sticlaru defined the class of minimal plus-one generated reduced plane curves explain their fundamental homological properties. Using their ideas and definitions, we can formulate the following main result of that note.
\begin{theoremA}
Let $\mathcal{L}\subset \mathbb{P}^{2}_{\mathbb{C}}$ be an arrangement of $d$ lines with only nodes and triple intersection points that is minimal plus-one generated. Then $\mathcal{L}$ has one of the following weak-combinatorics: 
$$(d;n_{2},n_{3}) \in \{(6;9,2),(7;9,4),(8;7,7),(9;6,10)\}.$$
\end{theoremA}
For all necessary definitions, please see Section $2$. This result is somewhat surprising, since the condition of being minimal plus-one generated turns out to be more restrictive compared to near-freeness \cite{Kabat}, or freeness itself \cite{FV}.
Our result fits well in the current investigations on plus-one generated arrangements of curves, see for instance \cite{MP}, where the authors study plus-one generated conic-line arrangements with nodes, tacnodes, and ordinary triple points, or the so-called maximizing line arrangements in the complex projective plane \cite{Liana}.

Let us present the structure of our paper. In Section $2$, we provide a short introduction into the world of plus-one generated curves. Then, in Section $3$, we provide some combinatorial constraints on plus-one generated line arrangements having double and triple intersections. In Section $4$, we provide our main result which is a classification of minimal plus-one generated line arrangements with only double and triple points. In the last section we study geometric and homological properties of specific arrangements with $9$ lines having $10$ triple and $6$ double points and, in particular, we construct a new example of the so-called weak Ziegler pairs.

In the paper we work over the complex numbers and all symbolic computations are preformed using \verb}SINGULAR} \cite{Singular}.

\section{Introduction to plus-one generated curves}
 Let $C$ be a reduced curve $\mathbb{P}^{2}_{\mathbb{C}}$ of degree $d$ given by $f \in S :=\mathbb{C}[x,y,z]$. We denote by $J_{f}$ the Jacobian ideal generated by the partials derivatives $\partial_{x}f, \, \partial_{y}f, \, \partial_{z}f$. Moreover, we denote by $r:={\rm mdr}(f)$ the minimal degree of a relation among the partial derivatives, i.e., the minimal non-trivial degree $r$ of a triple $(a,b,c) \in S_{r}^{3}$ such that 
$$a\cdot \partial_{x} f + b\cdot \partial_{y}f + c\cdot \partial_{z}f = 0.$$
We denote by $\mathfrak{m} = \langle x,y,z \rangle$ the irrelevant ideal. Consider the graded $S$-module $N(f) = I_{f} / J_{f}$, where $I_{f}$ is the saturation of $J_{f}$ with respect to $\mathfrak{m}$.
\begin{definition}
We say that a reduced plane curve $C$ is an $m$-syzygy curve when the associated Milnor algebra $M(f)$ has the following minimal graded free resolution:
$$0 \rightarrow \bigoplus_{i=1}^{m-2}S(-e_{i}) \rightarrow \bigoplus_{i=1}^{m}S(1-d - d_{i}) \rightarrow S^{3}(1-d)\rightarrow S \rightarrow M(f) \rightarrow 0$$
with $e_{1} \leq e_{2} \leq ... \leq e_{m-2}$ and $1\leq d_{1} \leq ... \leq d_{m}$. The $m$-tuple $(d_{1}, ..., d_{m})$ is called the exponents of $C$.
\end{definition}

\begin{definition}
A reduced curve $C$ in $\mathbb{P}^{2}_{\mathbb{C}}$ is called \textbf{plus-one generated} with the exponents $(d_1,d_2)$ if $C$ is $3$-syzygy such that $d_{1}+d_{2}=d$. 
\end{definition}
In order to study plus-one generated line arrangements we will use the following characterization that comes from \cite{DS0}. Here, by $\tau(C)$, we denote the total Tjurina number of a given reduced curve $C \subset \mathbb{P}^{2}_{\mathbb{C}}$ which is defined as the degree of the Jacobian ideal $J_{f}$.
\begin{proposition}[Dimca-Sticlaru]
\label{dimspl}
Let $C: f=0$ be a reduced $3$-syzygy curve of degree $d\geq 3$ with the exponents $(d_{1},d_{2},d_{3})$. Then $C$ is plus-one generated if and only if
$$\tau(C) = (d-1)^{2} - d_{1}(d-d_{1}-1) - (d_{3}-d_{2}+1).$$
\end{proposition}
\noindent
In order to define the main object of our interest, we will follow \cite{DS24}.
Define the $\delta$-level of a plus-one generated curve by
$$\delta L(C) = d_{3} - d_{2} \geq 0.$$

\begin{definition}
A plus-one generated curve $C$ satisfying $\delta L(C)=1$ is called a minimal plus-one generated curve.
\end{definition}
\begin{remark}
In the case where $\delta L(C) = 0$, we call $C$ a nearly free curve.
\end{remark}
For such minimal plus-one generated curves we have the following characterization that we are going to exploit in the paper.
\begin{theorem}[{\cite[Theorem 1.5]{DS24}}]
\label{MPOG}
Let $C \, : f=0$ be a reduced plane curve of degree $d$ in $\mathbb{P}^{2}_{\mathbb{C}}$ with $r = {\rm mdr}(f) \leq d/2$. Then $C$ is a minimal plus-one generated curve if and only if
$$r^{2} - r(d-1) + (d-1)^2 = \tau(C)+2.$$

\end{theorem}
\section{Combinatorial constraints on minimal plus-one generated line arrangements with double and triple intersection points}
In the first step we want to find an upper bound for the number of lines in our plus-one generated arrangements. We start with the following proposition that comes from \cite[Proposition 4.7]{DimcaPokora}.
\begin{proposition}
Let $C \, : \, f = 0$ be an arrangement of $d$ lines in $\mathbb{P}^{2}_{\mathbb{C}}$ such that it has only double and triple intersection points. Then one has
$${\rm mdr}(f) \geq \frac{2}{3}d - 2.$$
\end{proposition}
If $C \, : f=0$ is a plus-one generated arrangement of $d$ lines with double and triple intersection points with the exponents $(d_{1}, d_{2})$, $d_{1}\leq d_{2}$, then ${\rm mdr}(f) = d_{1}$. Since 
$$2d_{1} \leq d_{1}+d_{2} =d$$
we obtain that ${\rm mdr}(f) \leq d/2$. Combining it with the above proposition, we arrive at
$$\frac{2}{3}d - 2 \leq {\rm mdr}(f) \leq d/2.$$
This gives us the following result.
\begin{proposition}
If $\mathcal{L} \subset \mathbb{P}^{2}_{\mathbb{C}}$ is a plus-one generated arrangement of $d$ lines with double and triple intersection points, then $d\leq 12$. 
\end{proposition}
Having such a powerful bound on the number our plus-one generated arrangements, we can formulate the following natural problem.
\begin{problem}
Classify all weak-combinatorics of all plus-one generated line arrangements with double and triple intersection points in $\mathbb{P}^{2}_{\mathbb{C}}$
\end{problem}
Here by the weak-combinatorics, for an arrangement of lines $\mathcal{L}$, we mean the vector $(d;n_{2},n_{3})$, where $d$ is the number of lines, $n_{2}$ is the number of nodes, and $n_{3}$ is the number of triple points. Sometimes we will write $n_{i}(\mathcal{L})$ with $i \in \{2,3\}$ in order to emphasize the underlying arrangement of lines $\mathcal{L}$.
The first step towards the classification is the following proposition.
\begin{proposition}
\label{triple}
Let $\mathcal{L} \subset \mathbb{P}^{2}_{\mathbb{C}}$ be a minimal plus-one generated arrangement of $d$ lines with nodes and triple intersection points, then
\begin{equation}
\label{t3}
n_{3} \geq \frac{1}{4}\bigg(d^{2}-4d-5\bigg).
\end{equation}
\end{proposition}
\begin{proof}
Recall that if $\mathcal{L}$ is minimal plus-one generated with $r = {\rm mdr}(f)$, where $f \in \mathbb{C}[x,y,z]$ is the defining equation, then by Proposition \ref{dimspl} one has
$$r^{2} - r(d-1) + (d-1)^2 = \tau(\mathcal{L})+2.$$
We have the following combinatorial count
$$\binom{d}{2} = \frac{d(d-1)}{2} = n_{2} + 3n_{3}.$$
Since 
$\tau(\mathcal{L}) = n_{2} + 4n_{3} = \binom{d}{2}+n_{3}$,
we obtain
$$r^{2} - r(d-1) + (d-1)^2 = r^{2}-r(d-1)+d^{2}-2d+1 = \binom{d}{2}+n_{3}+2.$$
After simple manipulations, we arrive at
\begin{equation}
\label{disc}
2r^{2} - 2r(d-1) + d^{2}-3d-2-2n_{3}=0.   
\end{equation}
The above equation can have integer roots if the discriminant $\triangle_{r}= -4d^2 +16d + 20 +16n_{3}$ of \eqref{disc} satisfies $\triangle_{r} \geq 0$, so we get
$$n_{3} \geq \frac{1}{4}\bigg(d^{2}-4d-5\bigg).$$
\end{proof}
This result provides us a lower bound on the number of triple points for our plus-one generated arrangements. Moreover, we have a general upper-bound on the number of triple points for line arrangements (that does not depend on the plus-one generation) due to Sch\"onheim \cite{Sch}. Let us define the following number
$$U_{3}(d) := \bigg\lfloor \bigg\lfloor \frac{d-1}{2} \bigg\rfloor \cdot \frac{d}{3} \bigg\rfloor - \varepsilon(d),$$
where $\varepsilon(d)=1$ if $d = 5 \, {\rm mod}(6)$ and $\varepsilon(d) =0$ otherwise. Then
\begin{equation}
    n_{3} \leq U_{3}(d).
\end{equation}
If $\mathcal{L}$ is minimal plus-one generated with only double and triple intersections, then
\begin{equation}
\label{in1}
\frac{1}{4}\bigg(d^{2}-4d-5\bigg) \leq n_{3} \leq U_{3}(d).
\end{equation}
Using the above chain of inequalities, we can automatically conclude the following.
\begin{corollary}
Let $\mathcal{L} \subset \mathbb{P}^{2}_{\mathbb{C}}$ be a nearly free arrangement of $d$ lines with only nodes and triple intersection points. Then
$$d \leq 9.$$
\end{corollary}
\begin{proof}
    Observe that for $d \in \{10, 11, 12\}$ the chain of inequalities in (\ref{in1}) leads us to a contradiction. Indeed, if we take $d=12$, then
    $$22\frac{3}{4} \leq n_{3} \leq 20,$$
    and we get a contradiction. In a similar way we treat the cases $d=11$ and $d=10$.
\end{proof}

\section{Classification of minimal plus-one generated arrangements}
Using the above results providing constraints on our plus-one generated line arrangements, we can present our complete classification that is phrased in the language of their weak-combinatorics.
\begin{theorem}
Let $\mathcal{L}\subset \mathbb{P}^{2}_{\mathbb{C}}$ be an arrangement of $d$ lines with nodes and triple intersection points that is minimal plus-one generated. Then $\mathcal{L}$ has one of the following weak-combinatorics: 
$$(d;n_{2},n_{3}) \in \{(6;9,2),(7;9,4),(8;7,7),(9;6,10)\}.$$
In fact, all these weak-combinatorics can be geometrically realized over the reals.
\end{theorem}
\begin{proof} Our proof is divided into degree-dependent parts. Hereby we assume that $\mathcal{L} \, : f=0$ be is minimal plus-one generated with only double and triple intersections.
\begin{enumerate}
    \item[$d=4$:] Our first constraint is
    $$\frac{4}{2} \geq {\rm mdr}(f) \geq \frac{2}{3}\cdot 4-2,$$
    which implies that ${\rm mdr}(f) \in \{1,2\}$. Then the following system of equalities holds:\[
    \begin{cases}
      n_{2} + 4n_{3} =5  \\
      n_{2} + 3n_{3} = \binom{4}{2}=6
    \end{cases},\]
    where the first equation comes from Theorem \ref{MPOG} and a formula for the total Tjurina number, and the second is just the naive combinatorial count. Solving this system we get $t_{3}=-1$, a contradiction.
    \item[$d=5$:]  Observe that
    $$\frac{5}{2} \geq {\rm mdr}(f) \geq \frac{2}{3}\cdot 5-2,$$
    implies ${\rm mdr}(f) = 2$. Then the following system of equalities holds:\[
    \begin{cases}
      n_{2} + 4n_{3} = 10  \\
      n_{2} + 3n_{3} = \binom{5}{2}=10
    \end{cases}.\]
    Then $n_{3}=0$ and $n_{2}=10$. However, if an arrangement of $d$ lines has only nodes as singularities, then ${\rm mdr}(f) \geq d - 2$ (cf. \cite[Proposition 3.2]{MP}), so for $d=5$ we get ${\rm mdr}(f) \geq 3$, a contradiction. Hence for $d=5$ there is no minimal plus-one generated arrangement with double and triple points.
    \item[$d=6$:]  Similarly as in the previous case, we have
    $$\frac{6}{2} \geq {\rm mdr}(f) \geq \frac{2}{3}\cdot 6-2,$$
    which implies that ${\rm mdr}(f) \in \{2,3\}$. Then the following system of equalities holds:\[
    \begin{cases}
      n_{2} + 4n_{3} = 17  \\
      n_{2} + 3n_{3} = \binom{6}{2}=15
    \end{cases},\]
    which gives us that $n_{3}=2$ and $n_{2}=9$.
    By Proposition \ref{triple}, one has
    $$n_{3} \geq \frac{7}{4},$$
    hence $n_{3} \geq 2$ and we want to check whether the weak-combinatorics 
    $$(d;n_{2},n_{3}) = (6,9,2)$$
    can be realized geometrically. Consider an arrangement $\mathcal{L}_{6} \subset \mathbb{P}^{2}_{\mathbb{R}}$ that is given by
    $$Q(x,y,z) = y(x-z)(y-x-2z)(y+x-2z)(y-x+2z)(y+x+2z).$$
    Obviously $\mathcal{L}_{6}$ has $n_{2}=9$, $n_{3}=2$, and using \verb}SINGULAR} one can check that ${\rm mdr}(Q)=3$, hence $\mathcal{L}_{6}$ is a minimal plus-one generated arrangement.
    \item[$d=7$:] Analogously as above,
    $$\frac{7}{2} \geq {\rm mdr}(f) \geq \frac{2}{3}\cdot 7-2,$$
    which implies that ${\rm mdr}(f) =3$. Then the following system of equalities holds:\[
    \begin{cases}
      n_{2} + 4n_{3} = 25  \\
      n_{2} + 3n_{3} = \binom{7}{2}=21
    \end{cases},\]
    which gives us that $n_{3}=4$ and $n_{2}=9$.
    By Proposition \ref{triple}, one has
    $$n_{3} \geq 4,$$
    so we have to check whether the weak-combinatorics 
    $$(d;n_{2},n_{3}) = (7,9,4)$$
    can be realized geometrically. Consider an arrangement $\mathcal{L}_{7} \subset \mathbb{P}^{2}_{\mathbb{R}}$ that is given by
    $$Q(x,y,z) = xy(x-z)(y-x-2z)(x+y-2z)(y-x+2z)(x+y+2z).$$
    Obviously $\mathcal{L}_{7}$ has $n_{2}=9$, $n_{3}=2$, and using \verb}SINGULAR} one can check that ${\rm mdr}(Q)=3$, hence $\mathcal{L}_{7}$ is a minimal plus-one generated arrangement.
    \item[$d=8$:] We have
    $$\frac{8}{2} \geq {\rm mdr}(f) \geq \frac{2}{3}\cdot 8-2,$$
    which implies that ${\rm mdr}(f) =4$. Then the following system of equalities holds:\[
    \begin{cases}
      n_{2} + 4n_{3} = 35  \\
      n_{2} + 3n_{3} = \binom{8}{2}=28
    \end{cases},\]
    which gives us that $n_{3}=7$ and $n_{2}=7$.
    By Proposition \ref{triple}, one has
    $$n_{3} \geq 6,$$
    which means that we need to check whether the weak-combinatorics 
    $$(d;n_{2},n_{3}) = (8;7,7)$$
    can be realized geometrically. Consider an arrangement $\mathcal{L}_{8} \subset \mathbb{P}^{2}_{\mathbb{R}}$ that is given by
    $$Q(x,y,z) = y(y-2z)(y+2z)(y-x-2z)(y+x-2z)(y-x+2z)(y+x+2z)(y+x-6z).$$
    Obviously $\mathcal{L}_{8}$ has $n_{2}=7$, $n_{3}=7$, and using \verb}SINGULAR} one can check that ${\rm mdr}(Q)=4$, hence $\mathcal{L}_{8}$ is a minimal plus-one generated arrangement.
    \item[$d=9$:] We have
    $$\frac{9}{2} \geq {\rm mdr}(f) \geq \frac{2}{3}\cdot 9-2,$$
    which implies that ${\rm mdr}(f) = 4$. Then the following system of equations holds:\[
    \begin{cases}
      n_{2} + 4n_{3} = 46  \\
      n_{2} + 3n_{3} = \binom{9}{2}=36
    \end{cases},\]
    which gives us that $n_{3}=10$ and $n_{2}=6$.
    By Proposition \ref{triple} one has also that
    $$n_{3} \geq 10,$$
    which means that we have to check whether the weak-combinatorics 
    $$(d;n_{2},n_{3}) = (9;6,10)$$
    can be realized geometrically. Consider an arrangement $\mathcal{L}_{9} \subset \mathbb{P}^{2}_{\mathbb{R}}$ that is given by
    $$Q(x,y,z) = y(y-z)(y+z)(y-x)(y+x)(y-x-2z)(y-x+2z)(y+x-2z)(y+x+2z).$$
    Obviously $\mathcal{L}_{9}$ has $n_{2}=6$, $n_{3}=10$, and using \verb}SINGULAR} one can check that ${\rm mdr}(Q)=4$, hence $\mathcal{L}_{9}$ is a minimal plus-one generated arrangement.
\end{enumerate}
This completes our classification.
\end{proof}
Finishing this section, let us just comment one important property that $\mathcal{L}_{9}$ possesses, namely the realization space of $\mathcal{L}_{9}$ is one-dimensional\footnote{We thank Lukas K\"uhne for pointing this out.}. In other words, the following polynomial
\begin{multline}
\label{3net}
Q_{t}(x,y,z) = xyz(y+z)(x+y+z)(x+ty)((t-1)x+(t-1)y+tz)\\ ((t-1)x+tz)((t-1)x+t(t-1)y+t^2z)
\end{multline} 
with $t \in \mathbb{C} \setminus \bigg\{0,1, \frac{1}{2} \pm \frac{\sqrt{3}\iota}{2}\bigg\}$
defines a one-dimensional family of line arrangements $\mathcal{L}_{t}$ such for each admissible $t$ arrangements $\mathcal{L}_{t} : Q_{t}(x,y,z)=0$ have the same intersection lattice (= isomorphic intersection lattice) as $\mathcal{L}_{9}$. In fact, $\mathcal{L}_{t}$ has a $3$-net structure and due to this reason it has {\it non-trivial monodromy}, please consult \cite[Theorem 2.2]{M} for details.
\section{MPOG arrangements with $d=9$ lines}
In this short section we describe a new weak Ziegler pair that showed up during our investigations.
Our main result provides a complete classification of weak-combinatorics minimal plus-one generated line arrangements with only double and triple intersection points and it is very natural to wonder whether we can say more about geometric realizations of these arrangements. Here we want to focus on the weak-combinatorics $$(d;n_{2},n_{3}) = (9;6,10).$$
Using the database of matroids \cite{mat}, we can learn that there are exactly $4$ non-isomorphic matroids having the prescribed weak-combinatorics, but only two can be realized geometrically over the complex numbers. The first combinatorics can be realized by $\mathcal{L}_{9}$, defined in the previous section, and the second one is going to be described right now.
Let $\mathcal{L}_{9}' \subset \mathbb{P}^{2}_{\mathbb{C}}$ be the arrangement given by
$$Q(x,y,z) =  xyz(x+y)(x+z)(y-z)(x+y+ez)(x-ey+z)(x-ey+ez),$$
where $e^{2}+1=0$. The line arrangement $\mathcal{L}_{9}'$ has the same weak-combinatorics as $\mathcal{L}_{9}$, but these two line arrangements are not isomorphic since the realization space of $\mathcal{L}_{9}'$ is zero-dimensional, i.e. it consists of just two points determined by the equation $e^{2}+1=0$. Using \verb}SINGULAR} we can compute that ${\rm mdr}(\mathcal{L}_{9}')=5$ which means that $\mathcal{L}_{9}'$ cannot be minimal plus-one generated. This observation leads us to the following structural result.
\begin{theorem}
A line arrangement $\mathcal{L} \subset \mathbb{P}^{2}_{\mathbb{C}}$ with $d=9$ lines and only double and triple points as intersections is minimal plus-one generated if and only if $\mathcal{L}$ is the $3$-net $\mathcal{L}_{t}$ defined by equation \eqref{3net}.    
\end{theorem}
\begin{proof}
Since line arrangements $\mathcal{L}_{t}$ have the same weak-combinatorics and the total Tjurina number for line arrangements is determined by the weak-combinatorics, i.e., for every admissible $t$ we have $\tau(\mathcal{L}_{t})=46$, our problem boils down to showing that for every admissible $t$ one has ${\rm mdr}(\mathcal{L}_{t}) = 4$. This can be done using the following \verb}SINGULAR} routine, namely
\begin{verbatim}
ring R=(0,t),(x,y,z),(c,dp);
poly f = x*y*z*(y+z)*(x+y+z)*(x+t*y)*((t-1)*x+(t-1)*y+t*z)*((t-1)*x+t*z)
*((t-1)*x+t*(t-1)*y+t^2*z);
ideal I = jacob(f);
deg(syz(I)[1]);
\end{verbatim}
This procedure allows us to check that the resulting minimal syzygy is indeed of degree $4$ for each admissible parameter $t$. 

Using this preparation, we can finish our proof. Observe that one implication follows exactly from our discussion above since for every admissible $t$ arrangement $\mathcal{L}_{t}$ is minimal plus-one generated. The other implication follows from our discussion on the weak-combinatorics of arrangements with $d=9$ lines and only double and triple points. Since we have exactly two non-isomorphic realizations of the weak-combinatorics $(d;n_{2},n_{3}) = (9;6,10)$, and one realization is not minimal plus-one generated, this finishes the proof.
\end{proof}
Our last observation revolves around the notion of weak Ziegler pairs that were defined in \cite{Pokora}.

\begin{definition}[Weak Ziegler pair]
Consider two reduced plane curves $C_{1}, C_{2} \subset \mathbb{P}^{2}_{\mathbb{C}}$ such that all irreducible components of $C_{1}$ and $C_{2}$ are smooth. We say that a pair $(C_{1},C_{2})$ forms a \emph{weak Ziegler pair} if $C_{1}$ and $C_{2}$ have the same weak-combinatorics, but they have different minimal degrees of non-trivial Jacobian relations, i.e., ${\rm mdr}(C_{1}) \neq {\rm mdr}(C_{2})$.
\end{definition}
\begin{proposition}
Line arrangements $\mathcal{L}_{9}$ and $\mathcal{L}_{9}'$ form a weak Ziegler pair.
\end{proposition}
\begin{proof}
Recall that ${\rm mdr}(\mathcal{L}_{9})=4$ and ${\rm mdr}(\mathcal{L}_{9}')=5$, and this completes our justification.
\end{proof}
\section*{Acknowledgments}
The author would like to thank Piotr Pokora for his guidance.

Artur Bromboszcz is supported by the National Science Centre (Poland) Sonata Bis Grant  \textbf{2023/50/E/ST1/00025}. For the purpose of Open Access, the author has applied a CC-BY public copyright license to any Author Accepted Manuscript (AAM) version arising from this submission.

\vskip 0.5 cm

Artur Bromboszcz \\
Department of Mathematics,
University of the National Education Commission Krakow,
Podchor\c a\.zych 2,
PL-30-084 Krak\'ow, Poland. \\
\nopagebreak
\textit{E-mail address:} \texttt{artur.bromboszcz@uken.krakow.pl}
\bigskip
\end{document}